\documentclass[11pt]{amsart}

\usepackage{amsthm}
\usepackage{amsmath}
\usepackage{amssymb}

\newtheorem{thm}{Theorem}[section]
\newtheorem{theorem}[thm]{Theorem}

\newtheorem{corollary}[thm]{Corollary}
\newtheorem{problem}[thm]{Problem}
\newtheorem{notation}[thm]{Notation}

\newtheorem{lemma}[thm]{Lemma}
\newtheorem{proposition}[thm]{Proposition}

\newtheorem{definition}[thm]{Definition}
\theoremstyle{remark}

\newcommand{\norm}[1]{\|#1\|}

\newcommand{\RR}{\mathbb R}

\newcommand{\CC}{\mathbb C}

\newcommand{\cH}{\mathcal H}

\title{Fusion Frames and the Restricted Isometry Property}
\author[B.G. Bodmann]{Bernhard G. Bodmann}
\address{Department of Mathematics, University of Houston, Houston, TX 77204-3008}
\email{bgb@math.uh.edu}

\author[J. Cahill]{Jameson Cahill}
\address{Department of Mathematics, University
of Missouri, Columbia, MO 65211-4100}
\email{jameson.cahill@gmail.com}

\author[P.G. Casazza]{Peter G. Casazza} 
\address{Department of Mathematics, University
of Missouri, Columbia, MO 65211-4100}
\email{casazzap@missouri.edu}

\thanks{B. G. Bodmann was supported by NSF DMS 1109545 and 
by AFOSR FA9550-11-1-0245. J. Cahill and P. G. Casazza were supported by NSF DMS 1008183,
DTRA/NSF: ATD 1042701, AFOSR FA9550-11-1-0245}

\begin{document}

\maketitle

\begin{abstract}
We will show that tight frames satisfying the restricted isometry property 
give rise to nearly tight fusion frames which are nearly orthogonal and hence
are nearly equi-isoclinic.  We will also show how to replace parts of the
RIP frame with orthonormal sets while maintaining the RIP property.
\end{abstract}

\section{Introduction}

Fusion frames are a generalization of frames.
Fusion frames were introduced in \cite{CK04} under the name {\it frames of subspaces}
and quickly found application to problems in sensor networks, distributed processing and
more \cite{CF09,CK07b,CKL08,KPCL08}.  For a comprehensive view of the papers on fusion frames we refer the reader to
{\it www.fusionframes.org}.  While frames decompose a vector into scalar coefficients,
fusion frames decompose a vector into vector coefficients which can be locally processed
and later combined.  Fusion frames are designed to handle modern techniques for
information processing which today emphasizes distributed processing.  They allow
data processing to become a two step process where we first perform local processing
at individual nodes in the system and this is followed by integration of these results at
a central processor.  This has application to {\it packet-based network communications}, sensor networks,
radar imaging and more \cite{CCHKP11}.  This {\it hierarchical processing} helps to
design systems which are robust against noise, data loss, and erasures \cite{B,CK07b,
CKL08,KPCL08,KPC08}.  Much of the work on fusion frames has surrounded the
construction of fusion frames with specialized properties 
\cite{BCPST10,CFHWZ11,CFMWZ11,MRS}.

Our goal here is to use tools from compressed sensing,
namely matrices with the restricted isometry property, to
construct fusion frames with very strong properties.
Compressed sensing is a very hot topic today because of its broad application
to problems in sparse signal recovery.  There is so much literature in this area it
is not possible to adequately represent it here so we refer the reader to two recent
tutorials on the subject and their references \cite{FR,R}.  A fundamental tool in
this area is the {\it restricted isometry property (RIP)} (See section \ref{S1} for definitions).
This is a very powerful property for a family of vectors $\{\varphi_i\}_{i=1}^M$ in
$\cH_N$ which yields that subsets of a fixed size are {\it nearly orthonormal}.
As such, it is quite difficult to produce such families of vectors of the needed
sizes and they are constructed by probabilistic methods.
It is a fundamental open problem in the area to give a concrete construction
of RIP vectors of the appropriate sizes.  

In this paper, we will use tight frames of RIP matrices to construct fusion frames
with some very strong properties.  First we will show that we can construct
nearly tight fusion frames which still have the RIP property.  Next, we will construct
fusion frames with additional strong properties such as being {\it nearly equi-isoclinic}.
Finally, we will see how to replace subsets of our RIP family with orthonormal sequences
while tracking the change in the RIP constants.

\section{Frames and Fusion Frames}

Fusion frames are a generalization of frames.

\begin{definition}
A family of vectors $\{\varphi_i\}_{i\in I}$ in a Hilbert space $\cH$ is a {\it frame} for
$\cH$ if there are constants $0<A\le B < \infty$ so that for all $\varphi \in \cH$ we have
\[ A\|\varphi\|^2 \le \sum_{i\in I}|\langle \varphi,\varphi_i\rangle|^2 \le B \|\varphi\|^2.\]
\end{definition}
The numbers $A,B$ are {\it lower} (respectively, {\it upper}) frame bounds for the
frame.  If $A=B$ it is an $A${\it -tight frame} and if $A=B=1$, it is a {\it Parseval frame}.
If $\|\varphi_i\|=c$ for all $i\in I$ this is an {\it equal norm frame} and if $c=1$ it is
a {\it unit norm frame}.   The {\it analysis operator} of the frame is $T:\cH_N
\rightarrow \ell_2(M)$ given by
\[ T(\varphi) = \sum_{i=1}^M \langle \varphi,\varphi_i\rangle e_i,\]
where $\{e_i\}_{i=1}^M$ is the coordinate orthonormal basis of $\ell_2(M)$.
The {\it synthesis operator} is $T^*$ and is given by
\[ T^*\left ( \sum_{i=1}^Ma_ie_i\right ) = \sum_{i=1}^M a_i\varphi_i.\]
The {\it frame operator} is the positive self-adjoint invertible operator
$S=T^*T$ and satisfies
\[ S(\varphi) = \sum_{i=1}^M\langle \varphi,\varphi_i\rangle \varphi_i.\]
Reconstruction is given by
\[ \varphi = \sum_{i=1}^M \langle \varphi,\varphi_i\rangle S^{-1}\varphi_i=
\sum_{i=1}^M \langle \varphi,S^{-1/2}\varphi_i\rangle S^{-1/2}\varphi_i.\]
In particular, $\{S^{-1/2}\varphi_i\}_{i=1}^M$ is a Parseval frame for $\cH_N$.

Frame theory has application to a wide variety of problems in signal processing
and much more (see the monographs \cite{Gr, Ch} for a comprehensive view). 
Fusion frames are a generalization of frames and were introduced in
\cite{CK04}.  While frames decompose a signal into scalar coefficients,
fusion frames decompose signals into vector coefficients which can then be
locally processed and later combined to draw global conclusions.

\begin{definition}Given a Hilbert space $\cH$ and a family of closed subspaces $\{W_i\}_{i \in I}$ with
associated positive weights $v_i$, $i \in I$, a collection
of weighted subspaces $\{(\mathcal{W}_i,v_i)\}_{i \in I}$ 
is a {\em fusion frame} for $\cH$ if
there exist
constants $0 < A \le B < \infty$ satisfying
\[
A\|\varphi\|^2 \le \sum_{i \in I} v_i^2 \|P_i \varphi\|^2 \le B\|\varphi\|^2 \qquad \mbox{for any }\varphi \in \cH,
\]
where $P_i$ is the orthogonal projection onto $\mathcal{W}_i$. 
\end{definition}

The constants $A$ and $B$ are called
{\em fusion frame bounds}. A fusion frame is called {\em tight} if $A$ and $B$ can be chosen
to be equal, {\em Parseval} if $A=B=1$, and {\it orthonormal} if
\[ \cH = \oplus_{i\in I} W_i.\]  For $0<\epsilon<1$, the fusion frame is $\epsilon${\it -nearly tight}
if there is a constant $C$ so that $A=\frac{1}{1+\epsilon}C,\ B=(1+\epsilon)C$.
The fusion frame is {\it equi-dimensional} if all its subspaces $W_i$ have the same dimension.
\begin{notation}If $\{W_i\}_{i\in I}$ are subspaces of $\cH_N$, we define the space
\[\left ( \sum_{i\in I}\oplus W_i\right )_{\ell_2} = \{\{\psi_i\}_{i\in I}|\ \psi_i \in W_i\mbox{ and }
\sum_{i\in I}\|\psi_i\|^2 < \infty\},\]
with inner product given by
\[ \left \langle \{\psi_i\}_{i\in I},\{\tilde{\psi}_i\}_{i\in I}\right \rangle = \sum_{i\in I}\langle
\psi_i,\tilde{\psi}_i\rangle.\]
\end{notation}
The {\it analysis operator} of the fusion frame is the operator 
\[T:\cH_N \rightarrow \left ( \sum_{i\in I}\oplus W_i\right )_{\ell_2},\]
given by
\[ T(\varphi) = \{v_iP_i\varphi\}_{i\in I}.\]
The {\it synthesis operator} of the fusion frame is $T^*$ and is given by
\[ T^*\left ( \{\psi_i\}_{i\in I}\right ) = \sum_{i\in I} v_i\psi_i.\]

The fusion frame operator is the positive, self-adjoint and invertible operator $S_W:\cH \rightarrow \cH$ given by
\[ S_W\varphi = \sum_{i\in I}v_i^2P_i\varphi
,\mbox{ for all } \varphi \in \cH.\]  It is known \cite{CKL08} that $\{W_i,v_i\}_{i\in I}$ is a fusion frame with fusion frame bounds $A,B$ if and only if $AI \le S_W \le BI$.
Any signal $\varphi \in \cH$ can be reconstructed \cite{CKL08} from its fusion frame
measurements $\{v_i P_i \varphi\}_{i \in I}$ by performing
\[
\varphi = \sum_{i \in I} v_i S^{-1} (v_i P_i \varphi).
\]

A frame $\{\varphi_i\}_{i\in I}$ can be thought of as a fusion frame of one dimensional subspaces 
where $W_i = span\ \{\varphi_i\}$ for all $i\in I$.  The fusion frame is then $\{W_i,\|\varphi_i\|\}$. 
 A difference
between frames and fusion frames is that for frames, an input signal $\varphi \in \cH$ is represented
by a collection of scalar coefficients $\{\langle \varphi,\varphi_i\rangle \}_{i\in I}$ 
that measure the projection of
the signal onto each frame vector $\varphi_i$, while for fusion frames, an input signal $\varphi \in \cH$
is represented by a collection of {\it vector coefficients} $\{\Pi_{W_i}(\varphi)\}_{i\in I}$ corresponding to
projections onto each subspace $W_i$.

Much work has been put into the construction
of fusion frames with specified properties \cite{BCPST10,CFHWZ11,CFMWZ11}.  We also
have a generalization of fusion frames using non-orthogonal projections \cite{CCL10}.

There is an important connection between fusion frame bounds and bounds from frames taken from each
of the fusion frame's subspaces \cite{CK}.

\begin{theorem}
For each $i\in I$, let $v_i>0$ and $W_i$ be a closed subspace of $\cH$, and let
$\{\varphi_{ij}\}_{j\in J_i}$ be a frame for $W_i$ with frame bounds $A_i,B_i$.  Assume that
$0< A= \inf_{i\in I}A_i \le \sup_{i\in I}B_i=B<\infty$.  Then the following conditions hold:
\begin{enumerate}
	\item  $\{W_i,v_i\}_{i\in I}$ is a fusion frame for $\cH$.
	\item $\{v_i\varphi_{ij}\}_{i\in I,j\in J_i}$ is a frame for $\cH$.
\end{enumerate}
In particular, if $\{W_i,v_i\}_{j\in J_i}\}_{i\in I}$ is a fusion frame for $\cH$ with 
fusion frame bounds $C,D$, then $\{v_i\varphi_{ij}\}_{i\in I,j\in J_i}$ is a frame for $\cH$
with frame bounds $AC,BD$.  Also, if $\{v_i\varphi_{ij}\}_{i\in I,j\in J_i}$ is a frame for $\cH$ with frame
bounds $C,D$, then $\{W_i,v_i,\}_{j\in J_i}\}_{i\in I}$ is a fusion frame for $\cH$ with fusion
frame bounds $\frac{C}{B},\frac{D}{A}$.
\end{theorem}

\begin{corollary}
For each $i\in I$, let $v_i>0$ and $W_i$ be a closed subspace of $\cH$.  The following
are equivalent:
\begin{enumerate}
	\item $\{W_i,v_i\}_{i\in I}$ is a fusion frame for $\cH$ with fusion frame bounds $A,B$.
	\item For every orthonormal basis $\{e_{ij}\}_{j\in K_i}$ for $W_i$, the family $\{v_ie_{ij}\}_{i\in I,j\in K_i}$ is a frame for $\cH$ with frame bounds $A,B$.
	\item For every Parseval frame $\{\varphi_{ij}\}_{i\in I,j\in J_i}$ for 
	$W_i$, the family $\{v_i\varphi_{ij}\}_{i\in I,j\in J_i}$ is a frame for $\cH$ with frame bounds $A,B$.
\end{enumerate} 
\end{corollary}

\begin{corollary}
For each $i\in I$, let $v_i>0$ and $W_i$ be a closed subspace of $\cH$.  The following
are equivalent:
\begin{enumerate}
	\item $\{W_i,v_i\}_{i\in I}$ is a Parseval fusion frame for $\cH$.
	\item  For every orthonormal basis $\{e_{ij}\}_{j\in K_i}$ for $W_i$, the family $\{v_ie_{ij}\}_{i\in I,j\in K_i}$ is a Parseval frame for $\cH$.
	\item For every Parseval frame $\{\varphi_{ij}\}_{i\in I,j\in J_i}$ for $W_i$, the family 
	$\{v_i\varphi_{ij}\}_{i\in I,j\in J_i}$ is a Parseval frame for $\cH$.
\end{enumerate} 
\end{corollary}

\section{$\epsilon$-Riesz Sequences}

For our work we will need some information concerning $\epsilon$-Riesz sequences.

\begin{definition}  
A family of vectors $\{\varphi_i\}_{i=1}^N$ in $\cH_N$ is a Riesz basis with 
lower (resp. upper) Riesz bounds $0<A\le B< \infty$ if for all scalars $\{a_i\}_{i=1}^N$
we have
\[ A\sum_{i=1}^N|a_i|^2 \le \|\sum_{i=1}^Na_i\varphi_i\|^2 \le B \sum_{i=1}^N|a_i|^2.\]
This family of vectors is an $\epsilon$-Riesz basis for $\cH_N$
if for all scalars $\{a_i\}_{i=1}^N$ we have
\[ \frac{1}{1+\epsilon}\sum_{i=1}^N|a_i|^2 \le \|\sum_{i=1}^Na_i\varphi_i\|^2 \le
(1+\epsilon)\sum_{i=1}^N|a_i|^2.\]

The vectors are an $\epsilon$-Riesz sequence if they are an $\epsilon$-Riesz basis
for their span.
\end{definition}

As one can see, $\epsilon$-Riesz sequences are nearly orthonormal. 
The next few lemmas will formalize this statement.  First we recall that for a linearly
independent set of vectors $\{\varphi_i\}_{i=1}^N$ in $\cH_N$, the frame bounds of
this family equal the Riesz bounds.   It follows that if $S$ is the frame operator for
$\{\varphi_i\}_{i=1}^N$ then $\{S^{-1/2}\varphi_i\}_{i=1}^N$ is an orthonormal basis
for $\cH_N$.

\begin{proposition}\label{p1}
Let $\{\varphi_i\}_{i=1}^M$ be a family of unit norm vectors which is a $\epsilon$-Riesz
sequence.  Then for every partition $\{I_j\}_{j=1}^r$ of
$\{1,2,\ldots,M\}$  we have for all scalars $\{a_i\}_{i=1}^M$
\[ \frac{1}{(1+\epsilon)}\sum_{j=1}^r\|\sum_{i\in I_j}a_i\varphi_i\|^2 \le
 \sum_{i=1}^M |a_i|^2 \le (1+\epsilon)\sum_{j=1}^r\|\sum_{i\in I_j}
a_i\varphi_i\|^2.\]
Hence,
\[ \frac{1}{(1+\epsilon)^2}\sum_{j=1}^r\|\sum_{i\in I_j}a_i\varphi_i\|^2 \le
\|\sum_{i=1}^Ma_i\varphi_i\|^2 \le (1+\epsilon)^2 \sum_{j=1}^r\|\sum_{i\in I_j}a_i\varphi_i\|^2 .\]
\end{proposition}

\begin{proof}
We compute
\begin{eqnarray*}
\frac{1}{(1+\epsilon)}\sum_{j=1}^r\|\sum_{i\in I_j}a_i\varphi_i\|^2 &\le& 
\frac{1}{(1+\epsilon)}\sum_{j=1}^r (1+\epsilon)\sum_{i\in I_j}|a_i|^2\\
&=& \sum_{i\in \cup_{j=1}^r I_j}|a_i|^2\\
&=& \sum_{j=1}^r\sum_{i\in I_j}|a_i|^2\\
&\le&\sum_{j=1}^r (1+\epsilon)\|\sum_{i\in I_j}a_i\varphi_i\|^2\\
&=& (1+\epsilon)\sum_{j=1}^r\|\sum_{i\in I_j}a_i\varphi_i\|^2.
\end{eqnarray*}
For the hence, we combine the first part of the proposition with the fact that
\[ \frac{1}{1+\epsilon}\sum_{i=1}^M|a_i|^2 \le \|\sum_{i=1}^Ma_i\varphi_i\|^2
\le (1+\epsilon)\sum_{i=1}^M|a_i|^2.\]
\end{proof}

\begin{lemma}\label{l5}
If $\{\varphi_i\}_{i=1}^N$ is an $\epsilon$-Riesz basis for $\cH_N$ and let
$S$ be the frame operator.  Then
\[ \frac{1}{1+\epsilon}I \le S \le (1+\epsilon)I.\]
Hence,
\[ \frac{1}{1+\epsilon}I \le S^{-1} \le (1+\epsilon)I.\]
In general, if $0<a$ then
\[ \frac{1}{(1+\epsilon)^a}I\le S^a \le (1+\epsilon)^aI.\]
Hence, if $a>0$ then
\[ \frac{1}{(1+\epsilon)^a}I \le S^{-a} \le (1+\epsilon)^aI.\]
\end{lemma}

\begin{proof}
Let $T$ be the analysis operator for the Riesz basis.  By the definition,
for any scalars $\{a_i\}_{i=1}^N$ we have
\[ \|T^*(\{a_i\}_{i=1}^N)\|^2= \|\sum_{i=1}^Na_i\varphi_i\|^2 \le (1+\epsilon)\|\{a_i\}_{i=1}^N\|^2.
\]
And similarly,
\[ \|T^*(\{a_i\}_{i=1}^N)\|^2 \ge \frac{1}{1+\epsilon}\|\{a_i\}_{i=1}^N\|^2.\]
It follows that $T$ satisfies the same inequalities.
For any $\varphi \in \cH_N$ and any $0<a$ we have
\begin{eqnarray*}
\langle S^a\varphi, \varphi \rangle &=& \langle (T^*T)^a\varphi, \varphi \rangle\\
&=& \langle (T^*T)^{a/2}\varphi,(T^*T)^{a/2}\varphi\rangle\\
&=&\|(T^*T)^{a/2}\varphi\|^2\\
&\le& \|(T^*T)^{a/2}\|^2 \|\varphi\|^2\\
&=& \|T^*T\|^a\|\varphi\|^2\\
&\le& (1+\epsilon)^a\|\varphi\|^2.
\end{eqnarray*}
This shows that $S^a \le (1+\epsilon)^aI$.  The lower bound is derived similarly.
\end{proof}

Finally, we need to measure the {\it angle} between spaces spanned by disjoint
subsets of a $\epsilon$-Riesz sequence.

\begin{proposition}\label{pp1}
Let $\{\varphi_i\}_{i=1}^M$ be an $\epsilon$-Riesz sequence and choose any
partition $\{I_1,I_2\}$ of $\{1,2,\ldots,M\}$.  If $\varphi \in span\ \{\varphi_i\}_{i\in I_1}$
and $\psi \in span\ \{\varphi_i\}_{i\in I_2}$ are unit vectors, then
\[ |\langle \varphi,\psi\rangle|<2 \epsilon\Bigl(1+\frac \epsilon 2\Bigr) .\]
\end{proposition}

\begin{proof}
Let $\varphi = \sum_{i\in I_1}a_i\varphi_i$ and $\psi = \sum_{i\in I_2}a_i\varphi_i$ and
we compute
\begin{eqnarray*}
\frac{1}{1+\epsilon}\sum_{i=1}^M |a_i|^2 &\le& \|\varphi+\psi\|^2\\
&=& \|\varphi\|^2 + \|\psi\|^2 + 2 Re\langle \varphi,\psi\rangle\\
&\le& (1+\epsilon)\sum_{i=1}^M|a_i|^2.
\end{eqnarray*}
Hence,
\begin{eqnarray*}
2Re\langle \varphi,\psi\rangle &\le& (1+\epsilon)\sum_{i=1}^M|a_i|^2 - (\|\varphi\|^2
+\|\psi\|^2)\\
&\le& (1+\epsilon)\sum_{i=1}^M|a_i|^2 - (\frac{1}{1+\epsilon}\sum_{i\in I_1}|a_i|^2
+ \frac{1}{1+\epsilon}\sum_{i\in I_2}|a_i|^2 ) \\
&=& (1+\epsilon - \frac{1}{1+\epsilon}\sum_{i=1}^M|a_i|^2\\
&=& \epsilon \frac{2+\epsilon}{1+\epsilon} \sum_{i=1}^M|a_i|^2.
\end{eqnarray*}
Next, we observe that $|\langle \varphi,\psi\rangle| = \max_{|\lambda| = 1} Re \langle \varphi, \lambda \psi\rangle$.
Thus, we obtain together with Proposition~\ref{p1},
\begin{eqnarray*} |\langle \varphi,\psi\rangle| &\le& \epsilon(1+\frac{\epsilon}{2})
 \frac{1}{1+\epsilon}\left [ \sum_{i\in I_1}|a_i|^2+ \sum_{i\in I_2}|a_i|^2 \right ]\\
&\le& \epsilon (1+ \frac \epsilon 2)(\|\varphi\|^2 + \|\psi\|^2) = 2 \epsilon (1+ \frac \epsilon 2) .
\end{eqnarray*}

\end{proof}

\section{Fusion Frames and the Restricted Isometry Property}\label{S1}

In this section we will show how to use tight frames of vectors which have the
$\epsilon$-{\it restricted isometry property} to construct $\epsilon$-nearly tight fusion frames.

\begin{definition}
A family of vectors $\{\varphi_i\}_{i=1}^M$ in $\cH_N$ has the {\it restricted
isometry property} with constant $0<\epsilon<1$ for sets of size $s\le N$ if
for every $I\subset \{1,2,\ldots,M\}$ with $|I|\le s$, the family $\{\varphi_i\}_{i\in I}$
is an $\epsilon$-Riesz basis for its span.
\end{definition}

The restricted isometry property is one of the cornerstones of {\it compressed sensing}.
Compressed sensing is one of the most active area of research today and so we refer the
reader to the tutorials \cite{FR,R} and their references for a background in the area.  It is known that
the optimal $\epsilon$ above is on the order of 
\[ \epsilon \sim \frac{s}{N}log\frac{M}{s}.\]

Now we will see how tight frames of restricted isometry vectors with constant
$\epsilon$ will produce
nearly tight fusion frames.

\begin{theorem}\label{t1}
Let $\{\varphi_i\}_{i=1}^M$ be a unit norm tight frame for $\cH_N$ which has RIP with constant
$\epsilon$ for sets of size $s$. Then for any partition $\{I_j\}_{j=1}^K$ 
of $\{1,2,\ldots,M\}$ with $|I_j|\le s$ if we let
\[ W_j = span_{i\in I_j}\varphi_i,\]
then $\{W_j,1\}_{j=1}^K$ is a 
fusion frame with fusion frame bounds
\[ \frac{M}{(1+\epsilon)N},\ \frac{M(1+\epsilon)}{N}.\]
Moreover, if $L\subset \{1,2,\ldots,K\}$ and for $j\in L$ we have
$J_j \subset I_j$ with $\sum_{j=1}^K|J_j|\le s$ then for all scalars we
have
\[ \frac{1}{1+\epsilon}\sum_{j=1}^L\|\sum_{i\in J_j}a_i\varphi_i\|^2 \le \|\sum_{j=1}^L
\sum_{i\in J_j}a_i\varphi_i\|^2 \le (1+\epsilon)\sum_{j=1}^K\|\sum_{i\in J_j}
a_i\varphi_i\|^2.\]
\end{theorem}

To prove the theorem we need a lemma.

\begin{lemma}\label{lemma5}
Under the assumptions of the theorem, if $P_j$ is the orthogonal projection of
$\cH_N$ onto $W_j$, then for any $\varphi \in \cH_N$ we have:
\[    \frac{1}{1+\epsilon}\sum_{i\in I}|\langle \varphi,\varphi_i\rangle|^2
  \le \|P_j\varphi \|^2 \le (1+\epsilon)\sum_{i\in I}|\langle \varphi,\varphi_i\rangle|^2.  \]
  Hence,
  \[ \frac{M}{(1+\epsilon)N}\|\varphi\|^2 \le \sum_{j=1}^K\|P_j\varphi\|^2 \le (1+\epsilon)
  \frac{M}{N}\|\varphi\|^2.\]
\end{lemma}

\begin{proof}
Let $S$ be the frame operator:
\[ S\varphi = \sum_{i\in I}\langle \varphi,\varphi_i\rangle \varphi_i,\mbox{ for all }\varphi \in \cH_N.\]
Let $\{e_j\}_{j=1}^{M}$ be the eigenbasis for $S$ with eigenvalues 
\[ (1+\epsilon) \ge \lambda_1 \ge \cdots \ge \lambda_{|I|} \ge \frac{1}{1+\epsilon} \ge
0 \ge 0 \ge \cdots \ge 0.\]
Then
\[ P_j\varphi = \sum_{j=1}^{|I|}\langle \varphi,e_j\rangle e_j.\]
So
\[ \|P_j\varphi \|^2 = \sum_{i=1}^{|I|}|\langle \varphi,e_i\rangle|^2.\]
On the other hand,
\[ S\varphi = \sum_{j=1}^{|I|}\lambda_j \langle \varphi,e_j\rangle e_j,\]
and so
\[ \langle S\varphi,\varphi\rangle = \sum_{j=1}^{|I|}\lambda_j|\langle \varphi,e_j
\rangle|^2.\]
That is,
\begin{eqnarray*}
\|P_j\varphi\|^2 &=&
\sum_{j=1}^{|I|}|\langle \varphi,e_j\rangle |^2\\
&\le&(1+\epsilon) \sum_{j=1}^{|I|}\lambda_j|\langle \varphi,e_j\rangle|^2\\
&=& (1+\epsilon)\langle S\varphi,\varphi \rangle\\
&=& (1+\epsilon)\sum_{j=1}^{|I|}|\langle \varphi,\varphi_j \rangle|^2\\
\end{eqnarray*}
The other inequality is similar.

For the {\it hence}, we just observe that
\[\sum_{i=1}^M |\langle \varphi,\varphi_i\rangle|^2 = \frac{M}{N}\|\varphi\|^2.\]
\end{proof}

\noindent {\bf Proof of Theorem \ref{t1}}:

For each $j=1,2,\ldots,K$ let $P_j$ be the othogonal projection of $\cH_N$ onto
$W_j$.  Then by the Lemma \ref{lemma5}, for any $\varphi \in \cH_N$ we have:
\[ \sum_{j=1}^K \|P_j\varphi\|^2 \le (1+\epsilon)\sum_{j=1}^K\sum_{i\in I_j}|\langle \varphi,
\varphi_j\rangle|^2= (1+\epsilon)\sum_{i=1}^M|\langle \varphi,\varphi_i\rangle|^2
= (1+\epsilon)\frac{M}{N}\|\varphi\|^2.
\]
Similarly,
\[ \sum_{j=1}^K \|P_j\varphi\|^2 \ge \frac{1}{(1+\epsilon)}\sum_{j=1}^K\sum_{i\in I_j}|\langle \varphi,
\varphi_j\rangle|^2= \frac{1}{(1+\epsilon)}\sum_{i=1}^M|\langle \varphi,\varphi_i\rangle|^2
= \frac{1}{(1+\epsilon)}\frac{M}{N}\|\varphi\|^2.
\]
This completes the proof.
\vskip12pt

\section{Nearly Equi-Isoclinic Fusion Frames and the Restricted Isometry Property}

In this section, we will see how to use tight frames of vectors with the restricted isometry
property to construct nearly equi-isoclinic fusion frames.

\begin{definition}
Given two subspaces $W_1,W_2$ of a Hilbert space $\cH$ with dim $W_1=k\le$ dim $W_2=\ell$,
the {\em principal angles} $(\theta_1,\theta_2,\ldots \theta_k)$ between the subspaces are defined as follows:
The first principal angle is
\[ \theta_1 = \min\{\arccos |\langle \varphi,\psi\rangle|: \varphi \in S_{W_1},\psi\in S_{W_2}\}\]
where $S_{W_i}=\{\varphi \in W_i : \norm{\varphi}=1\}$. Two vectors $\varphi_1,\psi_1$ are called {\em principal vectors} if they give the minimum above.

The other principal angles and vectors are then defined recursively via
\[ \theta_i = min\{\arccos |\langle \varphi,\psi\rangle|: \varphi \in S_{W_1},\psi\in S_{W_2},  \mbox{ and }
\varphi \perp \varphi_j,\ \psi \perp \psi_j, 1\le j \le i-1\}.\]
\end{definition}

\begin{definition}
Two $k$-dimensional subspaces $W_1,W_2$ of a Hilbert space are isoclinic with parameter
$\lambda$, if the angle $\theta$ between any $\varphi\in W_1$ and its orthogonal projection
$P\varphi$ in $W_2$ is unique with $\cos^2\theta = \lambda$.

Multiple subspaces are equi-isoclinic if they are pairwise isoclinic with the same
parameter $\lambda$.
\end{definition}

An alternative definition is given in \cite{ET07} where two subspaces are called
isoclinic if the stationary values of the angles of two lines, one in each subspace,
are equal.  The geometric characterization given by Lemmens and Seidel \cite{LS} is that
when a sphere in one subspace is projected onto the other subspace, then
it remains a sphere, although the radius may change. This is all equivalent to the principal angles between the subspaces 
being identical.  

Much work has been done on finding the maximum number of equi-isoclinic subspaces
given the dimensions of the overall space and the subspaces (and often the
parameter $\lambda$).  Specifically, Seidel and Lemmens \cite{LS} give an upper bound on the
number of real equi-isoclinic subspaces and Hoggar \cite{H} generalizes this
to vector spaces over $\RR$ and $\CC$.

\begin{definition}
Two $K$-dimensional subspaces $W_1,W_2$ with associated orthogonal projections $P_1$ and
$P_2$ are isoclinic with parameter $\lambda \ge 0$ if 
$$
  P_1P_2 P_1 = \lambda P_1 \mbox{ and } P_2 P_1 P_2 = \lambda P_2 \, .
$$
A family of subspaces $\{W_j\}$ is $\epsilon$-nearly equi-isoclinic if there exists $\lambda \ge 0$ 
such that for every two subspaces $P_i$ and $P_j$, $i \ne j$,
$$
   (\lambda - \epsilon^2) P_1 \le  P_1 P_2 P_1 \le  (\lambda + \epsilon^2) P_1 \mbox{ and }   (\lambda - \epsilon^2) P_2 \le  P_2 P_1P_2  \le  (\lambda + \epsilon^2) P_2 \, .
$$
We will call a equi-dimensional fusion frame $\{W_i\}_{i=1}^K$ $\epsilon$-nearly equi-isoclinic
if its subspaces $\{W_i\}_{i=1}^K$ are $\epsilon$-nearly equi-isoclinic.
\end{definition}

It can be checked that a fusion frame $\{W_i,1\}_{i=1}^K$ is $\epsilon$-nearly equi-isoclinic
if and only if the squared cosines of the principal angles between any two of its subspaces are within $\epsilon^2$ of 
a fixed $\lambda$. 

A related property is:

\begin{definition}
A fusion frame $\{W_i,v_i\}_{i=1}^K$ is $\epsilon$-nearly orthogonal if whenever
we take unit vectors $\varphi \in W_i$ and $\psi \in W_j$ for $1\le i\not= j \le K$ 
we have $|\langle \varphi,
\psi\rangle |<\epsilon$.
\end{definition}

An $\epsilon$-nearly orthogonal fusion frame is $\epsilon$-nearly equi-isoclinic
by default in the sense that it satisfies the definition with $\lambda=0$.

\begin{theorem}
Let $\{\varphi_i\}_{i=1}^M$ be a unit norm tight frame for $\cH_N$ which has 
the restricted isometry property with constant
$\epsilon$ for sets of size $s$. Then for any partition $\{I_j\}_{j=1}^K$ 
of $\{1,2,\ldots,M\}$ with $|I_j|\le \frac{s}{2}$ if we let
\[ W_j = span_{i\in I_j}\varphi_i,\]
then $\{W_j,1\}_{j=1}^K$ is a $\epsilon$-tight fusion frame with fusion frame bounds
\[ \frac{M}{(1+\epsilon)N},\ \frac{M(1+\epsilon)}{N}.\]
Moreover, this is a $2\epsilon(1+\epsilon)^2$-nearly orthogonal fusion frame and hence it is
a $2\epsilon(1+\epsilon)^2$-nearly equi-isoclinic fusion frame.
\end{theorem}

\begin{proof}
The first part of the theorem is immediate by Theorem \ref{t1} and the {\it moreover} part
is immediate by Proposition \ref{pp1}.
\end{proof}

\section{The Restricted Isometry Property with Orthonormal Subsets}

A natural problem is the following:
\begin{problem}
Can we construct a family of vectors $\{\varphi_i\}_{i=1}^M$ in $\cH_N$ with the
restricted isometry property with constant $0<\epsilon <1$ for sets of size $s$
our of orthonormal bases for $\cH_N$?  Or, can they be constructed from
orthonormal sequences each having $s$ elements?
\end{problem}

We will now look at how we might try to alter a family of vectors with the RIP
property to a set which contains orthonormal sequences with $s$ vectors each.
We will need a lemma for this proof.

\begin{lemma}\label{lem1}
Let $W_1,W_2$ be subspaces of $\cH_N$ and let $T:W_1\rightarrow W_2$ 
be a surjection which satisfies
\[ \|\varphi-T\varphi\|^2 \le \epsilon \|\varphi\|^2,\mbox{ for all } \varphi\in W_1.\]
Let $P_1$ be the orthogonal projection of $\cH_N$ onto $W_1$.
Then
\[ \|\psi-P_1\psi\|^2 \le 4\frac{\epsilon}{( 1-\epsilon)^2}\|\psi\|^2, \mbox{ for all }\psi \in W_2.\]
Hence,
\[ \|P_1\psi\|^2 \ge (1- \frac{4\epsilon}{(1-\epsilon)^2})\|\psi\|^2.\]

\end{lemma}

\begin{proof} 
First note that for any $\varphi \in W_1$ 
\[ (1-\epsilon)^2\|\varphi\|^2 \le \|T\varphi\|^2 \le (1+\epsilon)^2\|\varphi\|^2.\]
Next we have for any $\varphi \in W_1$
\[ \|\varphi-T\varphi\|^2 = \|\varphi - P_1T\varphi\|^2=\|P_1(I-T)\varphi\|^2 + \|(I-P_1)(I-T)\varphi\|^2
\le \epsilon \|\varphi\|^2.\]
Let $\psi \in W_2$.  Choose $\varphi \in W_1$ so that $T\varphi = \psi$.  Now we
compute
\begin{eqnarray*}
\|\psi-P_1\psi\| &=& \|\psi-P_1T\varphi\|\\
&\le& \|\psi - \varphi\| + \|\varphi - P_1T\varphi\|\\
&\le& \|T\varphi - \varphi\|+\|\varphi-P_1T\varphi\|\\
&\le& \sqrt{\epsilon}\|\varphi\| + \sqrt{\epsilon}\|\varphi\|\\
&\le& 2\sqrt{\epsilon}\|T^{-1}\psi\|\\
&\le& 2\sqrt{\epsilon}\|T^{-1}\|\|\psi\|\\
&\le& 2\frac{\sqrt{\epsilon}}{1-\epsilon}\|\psi\|.
\end{eqnarray*}

For the hence, we note that by Pythagoras
\begin{eqnarray*}
\|P_1\psi\|^2 &=&  \|\psi\|^2 - \|(I-P_1)\psi\|^2 \\
&\ge& (1- \frac{4\epsilon}{(1-\epsilon)^2})\|\psi\|^2 \, .
\end{eqnarray*}
\end{proof}

Now we are ready for the construction of RIP families which contain orthonormal
sets.

\begin{theorem}
Let $\{\varphi_i\}_{i=1}^M$ be a family of vectors in $\cH_N$ having the restricted isometry
property with constant $0<\epsilon <1$ for sets of size $s$.  Partition $\{1,2,\ldots,M\}$
into sets $\{I_j\}_{j=1}^K$ with $|I_j|\le s$ for all $j=1,2,\ldots,K$.   For each $j$ let
$S_j$ be the frame operator for $\{\varphi_i\}_{i\in I_j}$.  For $K_1 \le K$, replace 
for ach $j\le K_1$ the family
$\{\varphi_i\}_{i\in I_j}$ by $\{S_j^{-1/2}\varphi_i\}_{i\in I_j}$, which is an orthonormal
basis for its span.  Then $\{S_j^{-1/2}\varphi_i\}_{i\in I_j;j=1,2,\ldots,K_1} \cup
\{\varphi_i\}_{i\in I_j:K_1+1\le j \le K}=:\{\psi_i\}_{i=1}^M$ has the restricted isometry property 
and for sets $J \subset \{1,2,\ldots,M\}$ with $|J|\le s$ we have for all families of
scalars $\{a_i\}_{i\in J}$,
\[    \left [ \frac{1-4\epsilon/(1-\epsilon)^2}{(1+\epsilon)^2} - 4\epsilon(1+\epsilon)\sqrt{K_1}\right ] 
\left ( \sum_{i\in J}|a_i|^2 \right )^{1/2}      \]
\[\le             \|\sum_{i\in J}a_i\psi_i\|\le 
\left [ ((1+\epsilon)^{3/2}+4\epsilon(1+\epsilon)\sqrt{K_1}\right ] 
\left ( \sum_{i\in J}|a_i|^2 \right )^{1/2}.
\] 
\end{theorem}

\begin{proof}
Choose a subset $J\subset \{1,2,\ldots,M$ with $|J|\le s$ and let $J_j=J\cap I_j$ for all $j=1,2,\ldots,K$.
For each $1\le j \le K_1$ let $P_j$ be the orthogonal projection of $\cH_N$ onto
span $\{S_j^{-1/2}\varphi_i\}_{i\in J_j}$.  Choose any scalars $\{a_i\}_{i\in J_j:j=1,2\ldots,K}$.
Then
\begin{equation}\label{eqn4}
 \|\sum_{j=1}^{K_1}P_j\sum_{i\in J_j}a_iS_j^{-1/2}\varphi_i + \sum_{j=K_1+1}^{K}
\sum_{i\in J_j}a_i\varphi_i\|- \|\sum_{j=1}^{K_1}(I-P_j)\sum_{i\in J_j}a_iS_j^{-1/2}\varphi_i\|
\end{equation}
\[ \le \|\sum_{j=1}^{K_1}\sum_{i\in J_j}a_iS^{-1/2}\varphi_i + \sum_{j=K_1+1}^{K}
\sum_{i\in J_j}a_i\varphi_i\|\]
\[ \le \|\sum_{j=1}^{K_1}P_j\sum_{i\in J_j}a_iS_j^{-1/2}\varphi_i + \sum_{j=K_1+1}^{K}
\sum_{i\in J_j}a_i\varphi_i\|+ \|\sum_{j=1}^{K_1}(I-P_j)\sum_{i\in J_j}a_iS_j^{-1/2}\varphi_i\|
\]
We will consider the above two sums separately.  By Lemma \ref{l5} we have
\[ (I-S_j^{-1/2})^2 \le \left ( 1-\frac{1}{\sqrt{1+\epsilon}}\right )^2I\le \frac{\epsilon}{1+\epsilon}I.\]
Applying Lemma \ref{l5}
and using $T=S^{-1/2}$ in Lemma \ref{lem1} we have for all $j=1,2,\ldots,K_1$
\begin{eqnarray*}
 \|(I-P_j)\sum_{i\in J_j}a_iS_j^{-1/2}\varphi_i\| &\le& 
\frac{4\frac{\epsilon}{1+\epsilon}}{(1-\frac{\epsilon}{1+\epsilon})^2} 
\left ( \sum_{i\in J_j}|a_i|^2 \right )^{1/2}\\
&=& 4\epsilon(1+\epsilon)
\left ( \sum_{i\in J_j}|a_i|^2 \right )^{1/2}.
\end{eqnarray*}
Hence,
\begin{eqnarray}\label{eqn1}
 \|\sum_{j=1}^{K_1}(I-P_j)\sum_{i\in J_j}a_iS_j^{-1/2}\varphi_i\| 
 &\le&
 \sum_{j=1}^{K_1}   \|(I-P_j)\sum_{i\in J_j}a_iS_j^{-1/2}\varphi_i\| \\
 \nonumber
 &\le& 4\epsilon(1+\epsilon)\sum_{j=1}^{K_1}\left ( \sum_{i\in J_j}|a_i|^2 \right )^{1/2}\\
 \nonumber
 &\le& 4\epsilon(1+\epsilon)\sqrt{K_1}\left ( \sum_{j=1}^{K_1}\sum_{i\in J_j}|a_i|^2\right )^{1/2}
\end{eqnarray}
For the second term, since the vector
\[ \sum_{j=1}^{K_1}P_j\sum_{i\in J_j}a_iS_j^{-1/2}\varphi_i + \sum_{j=K_1+1}^{K}
\sum_{i\in J_j}a_i\varphi_i,\]
is contained in the span of the vectors $\{\varphi_i\}_{i\in J_j:j=1,2,\ldots,K}$ and
\[ \sum_{j=1}^K|J_j|=|J| \le s,\]
which is an $\epsilon$-Riesz sequence, we have by Proposition \ref{p1}
\[\frac{1}{(1+\epsilon)^2}\left [ \sum_{j=1}^{K_1}\|P_j\sum_{i\in J_j}a_iS_j^{-1/2}\varphi_i\|^2
+ \sum_{j=K_1+1}^K\|\sum_{i\in J_j}a_i\varphi_i\|^2\right ]\]
\[ \le \|\sum_{j=1}^{K_1}P_j\sum_{i\in J_j}a_iS_j^{-1/2}\varphi_i + \sum_{j=K_1+1}^{K}
\sum_{i\in J_j}a_i\varphi_i\|^2\]
\[
\le (1+\epsilon)^2 \left [ \sum_{j=1}^{K_1}\|P_j\sum_{i\in J_j}a_iS_j^{-1/2}\varphi_i\|^2
+ \sum_{j=K_1+1}^K\|\sum_{i\in J_j}a_i\varphi_i\|^2\right ]\]
Since $\{S_j^{-1/2}\varphi_i\}_{i\in J_j}$ is an orthonormal set, we have
\begin{equation}\label{eqn5}
\sum_{j=1}^{K_1}\|P_j\sum_{i\in J_j}a_iS_j^{-1/2}\varphi_i\|^2
+ \sum_{j=K_1+1}^K\|\sum_{i\in J_j}a_i\varphi_i\|^2
\end{equation}
\[ \le \sum_{j=1}^{K_1}\|\sum_{i\in J_j}a_iS_j^{-1/2}\varphi_i\|^2 + 
(1+\epsilon)\sum_{j=K_1+1}^K\sum_{i\in J_j}|a_i|^2\]
\[ = \sum_{j=1}^{K_1}\sum_{i\in J_j}|a_i|^2 +(1+\epsilon)\sum_{j=K_1+1}^K\sum_{i\in J_j}|a_i|^2\]
\[ \le (1+\epsilon)\sum_{i\in J}|a_i|^2.\]
Similarly, applying the {\it hence} from Lemma \ref{lem1} we have
\begin{equation}\label{eqn2}
 \sum_{j=1}^{K_1}\|P_j\sum_{i\in J_j}a_iS_j^{-1/2}\varphi_i\|^2
+ \sum_{j=K_1+1}^K\|\sum_{i\in J_j}a_i\varphi_i\|^2
\end{equation}
\[ \ge(1-4\epsilon/(1-\epsilon)^2)
 \sum_{j=1}^{K_1}\|\sum_{i\in J_j}a_iS_j^{-1/2}\varphi_i\|^2 + 
\frac{1}{(1+\epsilon)}\sum_{j=K_1+1}^K\sum_{i\in J_j}|a_i|^2\]
\[ = (1-4\epsilon/(1-\epsilon)^2)\sum_{j=1}^{K_1}\sum_{i\in J_j}|a_i|^2 +\frac{1}{(1+\epsilon)}
\sum_{j=K_1+1}^K\sum_{i\in J_j}|a_i|^2\]
\[ \ge (1-4\epsilon/(1-\epsilon)^2)\sum_{i\in J}|a_i|^2.\]
Putting this second part together we have
\[  \|\sum_{j=1}^{K_1}P_j\sum_{i\in J_j}a_iS_j^{-1/2}\varphi_i + \sum_{j=K_1+1}^{K}
\sum_{i\in J_j}a_i\varphi_i\|^2\]
\[ \le (1+\epsilon)^2 \left [ \sum_{j=1}^{K_1}\|P_j\sum_{i\in J_j}a_iS_j^{-1/2}\varphi_i\|^2
+ \sum_{j=K_1+1}^K\|\sum_{i\in J_j}a_i\varphi_i\|^2\right ]\]
\[ \le (1+\epsilon)^3 \sum_{i\in J}|a_i|^2.\]
Similarly,
\begin{equation}\label{eqn3}
 \|\sum_{j=1}^{K_1}P_j\sum_{i\in J_j}a_iS_j^{-1/2}\varphi_i + \sum_{j=K_1+1}^{K}
\sum_{i\in J_j}a_i\varphi_i\|^2
\end{equation}
\[ \ge \frac{1}{(1+\epsilon)^2} \left [ \sum_{j=1}^{K_1}\|P_j\sum_{i\in J_j}a_iS_j^{-1/2}\varphi_i\|^2
+ \sum_{j=K_1+1}^K\|\sum_{i\in J_j}a_i\varphi_i\|^2\right ]\]
And by equation \ref{eqn2} we can continue this inequality to
\[\ge \frac{1-4\epsilon/(1-\epsilon)^2}{(1+\epsilon)^2}\sum_{i\in J}|a_i|^2.\]
Finally, combining equations \ref{eqn4}, \ref{eqn1}, and \ref{eqn5} we have:
\[ \|\sum_{j=1}^{K_1}\sum_{i\in J_j}a_iS^{-1/2}\varphi_i + \sum_{j=K_1+1}^{K}
\sum_{i\in J_j}a_i\varphi_i\|\]
\[ \le \|\sum_{j=1}^{K_1}P_j\sum_{i\in J_j}a_iS_j^{-1/2}\varphi_i + \sum_{j=K_1+1}^{K}
\sum_{i\in J_j}a_i\varphi_i\|+ \|\sum_{j=1}^{K_1}(I-P_j)\sum_{i\in J_j}a_iS_j^{-1/2}\varphi_i\|\]
\[ \le (1+\epsilon)^{3/2} \left ( \sum_{i\in J}|a_i|^2\right )^{1/2} +
4\epsilon(1+\epsilon)\sqrt{K_1}\left ( \sum_{j=1}^{K_1}\sum_{i\in J_j}|a_i|^2 \right )^{1/2}\]
\[ \le \left [ ((1+\epsilon)^{3/2}+4\epsilon(1+\epsilon)\sqrt{K_1}\right ] 
\left ( \sum_{i\in J}|a_i|^2 \right )^{1/2}.\]

Similarly, combining equations \ref{eqn4}, \ref{eqn2} and \ref{eqn3} we have
\[ \|\sum_{j=1}^{K_1}\sum_{i\in J_j}a_iS^{-1/2}\varphi_i + \sum_{j=K_1+1}^{K}
\sum_{i\in J_j}a_i\varphi_i\|\]
\[ \ge  \|\sum_{j=1}^{K_1}P_j\sum_{i\in J_j}a_iS_j^{-1/2}\varphi_i + \sum_{j=K_1+1}^{K}
\sum_{i\in J_j}a_i\varphi_i\|- \|\sum_{j=1}^{K_1}(I-P_j)\sum_{i\in J_j}a_iS_j^{-1/2}\varphi_i\|\]
\[ \ge \frac{1-4\epsilon/(1-\epsilon)^2}{(1+\epsilon)^2}\left (\sum_{i\in J}|a_i|^2 \right )^{1/2}
- 4\epsilon(1+\epsilon)\sqrt{K_1}\left ( \sum_{j=1}^{K_1}\sum_{i\in J_j}|a_i|^2 
\right )^{1/2} \]
\[ \ge \left [ \frac{1-4\epsilon/(1-\epsilon)^2}{(1+\epsilon)^2} - 4\epsilon(1+\epsilon)\sqrt{K_1}\right ] 
\left ( \sum_{i\in J}|a_i|^2 \right )^{1/2}.\]
\end{proof}

So we can maintain the restricted isometry property after replacement of some
$K_1$ groups of $s$ vectors in the RIP family by orthonormal sets as long as
\[ 0<  \left [ \frac{1-4\epsilon/(1-\epsilon)^2}{(1+\epsilon)^2} - 4\epsilon(1+\epsilon)\sqrt{K_1}\right ] \]
Solving for $K_1$ we have
$$
   K_1 < \frac{1}{16 \epsilon^2}\frac{(1-4\epsilon/(1-\epsilon)^2)^2}{(1+\epsilon)^6}.
$$
So for sufficiently small $\epsilon$, the fraction on the right hand side is close to one
and we can let $K_1$ grow like $1/\epsilon^2$.

\end{document}